\documentclass[review]{elsarticle}
\usepackage{amsmath,amsthm,amssymb} 
\usepackage{enumitem}
\usepackage{tikz-cd}

\newtheorem{theo}{Theorem}[section]
\newtheorem{cor}[theo]{Corollary}

\newtheorem{lm}[theo]{Lemma}
\newtheorem{rem}[theo]{Remark}
\newtheorem*{remark}{Remark}
\newcommand{\tr}{\triangle}

\numberwithin{equation}{section}

\begin{document}	
\title{On extensions of hook Weyl modules}	

\author{Mihalis Maliakas\corref{cor1}}
\ead{mmaliak@math.uoa.gr}
\address{Department of Mathematics, University of Athens}
\author{Dimitra-Dionysia Stergiopoulou\fnref{fn2}}
\ead{dstergiop@math.uoa.gr}
\address{Department of Mathematics, University of Athens}

\cortext[cor1]{Corresponding author}
\fntext[fn2]{Partially supported by Onassis Foundation grant GZM 065-1.}

%\begin{keyword}Extensions, general linear group, Weyl module, hook
%\end{keyword}

%\subjclass[2010]{Primary 20G05;
%	Secondary 20G43}

%\keywords{Extensions, general linear group, hook Weyl module}

%\date{May 4, 2020}

\begin{abstract}
We determine the integral extension groups $\mathrm{Ext}^1(\Delta(\mathrm{h}),\Delta(\mathrm{h}(k)))$ and $\mathrm{Ext}^k(\Delta(\mathrm{h}),$ $\Delta(\mathrm{h}(k)))$,  where $\Delta(\mathrm{h}),\Delta(\mathrm{h}(k))$ are the Weyl modules of the general linear group $GL_n$  corresponding to hook partitions $\mathrm{h}=(a,1^b)$, $\mathrm{h}(k)=(a+k,1^{b-k})$.
\end{abstract}	

\maketitle

\section{Introduction}This paper concerns polynomial representations of the general linear group $GL_n$ over the integers. For a partition $\lambda$, let $\Delta(\lambda)$ denote the Weyl module of $GL_n$ of highest weight $\lambda$. The extension groups $\mathrm{Ext}^i(\Delta(\lambda),\Delta(\mu))$ play an important role in the theory. For example, the $p$-torsion of $\mathrm{Ext}^1(\Delta(\lambda),\Delta(\mu))$ yields the $\mathrm{Hom}$ space between the corresponding modular Weyl modules of $GL_n(K)$, where $K$ is an algebraically closed field of characteristic $p>0$, and the dimensions of the higher modular extensions may be obtained through torsion and restriction of integral extensions. Jantzen's sum formula can be viewed and proved via integral extension groups \cite{AK}. 

There are not many cases where explicit computations of integral extension groups between Weyl modules have been carried out. In \cite{AB} the $GL_2$ case was treated and in \cite{BF} the $GL_3$ case when $\lambda$ and $\mu$ differ by a multiple of a single root, both for $i=1$. In \cite{Ak} the case $\lambda=(1^a), \mu=(a)$ was studied and in \cite{Ma} the situation where $\lambda, \mu$ are hooks differing by a single root was considered. In \cite{Ku2} the case where $\lambda, \mu$ are any partitions differing by a single root was settled. As the modular extension groups are intimately related to the integral ones, we mention the result on neighboring Weyl modules \cite{Ja} Part II Section 7,  and the $SL_2$ result in \cite{Pa} for all $i$ generalizing \cite{Er} and \cite{CE}. More can be found in \cite{CP}.

Let $\mathrm{h}, \mathrm{h}(k)$ be hooks, $\mathrm{h}=(a,1^b), \mathrm{h}(k)=(a+k,1^{b-k})$, where $k$ is an integer such that $1 \le k \le b$. It follows  that $\mathrm{Ext}^i(\Delta(\mathrm{h})), \Delta(\mathrm{h}(k)))=0$ if $ i>k $. In this paper we determine $\mathrm{Ext}^1(\Delta(\mathrm{h}),\Delta(\mathrm{h}(k)))$ and $\mathrm{Ext}^k(\Delta(\mathrm{h})), \Delta(\mathrm{h}(k)))$ (Theorem 3.5 and Theorem 4.1). Our approach utilizes presentation matrices for various $\mathrm{Ext}$ groups that we determine from the description of generators and relations of Weyl modules of \cite{ABW} ($ i=1 $) and from the projective resolutions of \cite{Ma} ($ i=k $). Using these and the degree reduction theorem of Kulkarni \cite{Ku1}, we identify cyclic generators of extension groups of the form $\mathrm{Ext}^i(\Delta(\mathrm{h}),M)$, $ (i=1, k) $, where $ M $ is a tensor product of a divided power and an exterior power. Computing the image of these generators under canonical maps yields the results.

\section{Recollections}

\subsection{Notation} Let $F$ be a free abelian group of finite rank $n$.  Fixing a basis of $F$ yields an identification of general linear groups $GL(F)=GL_n(\mathbb{Z})$. We will be working with homogeneous polynomial representations of $GL_n(\mathbb{Z})$ of degree $r$, or equivalently, with modules over the Schur algebra $S_\mathbb{Z}(n,r)$  \cite{Gr}, Section 2.4. We will write $S(n,r)$ in place of $S_\mathbb{Z}(n,r)$.
By $DF=\sum_{i\geq 0}D_iF$ and $\wedge F=\sum_{i\geq 0}\wedge^{i}F$
we denote the divided power algebra of $F$ and the exterior algebra of $F$ respectively. We will usually omit $F$ and write $D_i$ and $\wedge^i$. 

From \cite{Gr} or \cite{AB}, Proposition 2.1, we recall that for each sequence $a_1,...,a_n$ of non negative integers $a_i$ that sum to $r$, the $S(n,r)$-module $D_{a_1} \otimes ... \otimes D_{a_n}$ is projective. Throughout this paper all tensor products are over the integers.

For a partition $\lambda$ of $r$ with at most $n$ parts, we denote by $\Delta(\lambda)$ the corresponding Weyl module for $S(n,r)$. If $\lambda =(a)$ is a partition with one part, then $\Delta(\lambda) =D_a$, and if $\lambda =(1^b)$, then $\Delta(\lambda) =\wedge^{b}$. A hook $\mathrm{h}$ is a partition of the form $\mathrm{h}=(a,1^b)$. The following complex of $S(n,r)$-modules (which is the dual of the usual Koszul complex) is exact 
$$0 \rightarrow D_{a+b} \rightarrow ... \rightarrow D_{a+1} \otimes \wedge^{b-1} \xrightarrow{\theta_{a}}  D_{a} \otimes \wedge^{b} \rightarrow ... \rightarrow \wedge^{a+b} \rightarrow 0,$$
where $\theta_{a}$ is the composition 
$\ D_{a+1}\otimes \wedge^{b-1} \xrightarrow{\triangle \otimes 1} D_{a}\otimes D_1 \otimes \wedge^{b-1}\xrightarrow{1 \otimes m} D_{a} \otimes \wedge^{b},$
where $\triangle$ (respectively, $m$) is the indicated component of the comultiplication (resp., multiplication) map of the Hopf algebra $DF$ (resp., $\wedge F $). It is  well known that if $\mathrm{h}=(a,1^b)$ is a hook, $b \geq 1$, then
$\Delta (\mathrm{h}) \simeq cok(\theta_{a}) \simeq ker(\theta_{a-1}),$ 
so that we have the following short exact sequence 
\begin{equation}0 \rightarrow \Delta (\mathrm{h}(1)) \xrightarrow{i} D_{a} \otimes \wedge^{b} \xrightarrow{\pi_0} \Delta(\mathrm{h}) \rightarrow 0, \end{equation}
where $\mathrm{h}(1)=(a+1,1^{b-1})$, the map $i$ is induced by $\theta_{a}$ on generators and the map $\pi_0$ is induced by the identity map on generators

\textit{Notation}: Throughout this paper we use the notation $\mathrm{h}=(a,1^b), \mathrm{h}(k)=(a+k,1^{b-k}),  1\le k\le b, r=a+b.$

\subsection{Straightening law} We recall the straightening law and the semi-standard basis theorem for $\Delta(\mathrm{h})$ (\cite{ABW}, Theorem II.3.16). Fix an ordered basis $e_1,..,e_n$ of $F$. For simplicity, we denote the element $e_i$ by $i$ and accordingly the element $e_{i_1}^{(a_1)}  ...  e_{i_t}^{(a_t)} \otimes e_{j_1} \wedge ... \wedge e_{j_b} \in D_a \otimes \wedge^b$ by ${i_1}^{(a_1)}  ...  {i_t}^{(a_t)} \otimes {j_1} ... {j_b}$. The image of this element under the identification $\Delta(\mathrm{h}) \simeq cok(\theta_{a})$ will be denoted by ${i_1}^{(a_1)}  ...  {i_t}^{(a_t)} | {j_1} ... {j_b}$. Now suppose $i_1<i_2<...<i_t$ and $j_1 \le i_1$. Then in $\Delta(\mathrm{h})$ we have 
$${i_1}^{(a_1)}...{i_t}^{(a_t)} | {j_1} ... {j_b} = \begin{cases} -\sum\limits_{s\geq 2}{i_1}^{(a_1+1)}  ...{i_s}^{(a_s-1)}...{i_t}^{(a_t)} | {i_s}{j_2} ... {j_b}, & \mbox{if}\;j_1=i_1 \\ -\sum\limits_{s\geq 1}j_1{i_1}^{(a_1)}  ...{i_s}^{(a_s-1)}...{i_t}^{(a_t)} | {i_s}{j_2} ... {j_b}, & \mbox {if}\;j_1<i_1. \end{cases} $$

A $\mathbb{Z}$ - basis of $\Delta(\mathrm{h})$ is the set of all ${i_1}^{(a_1)}  ...  {i_t}^{(a_t)} | {j_1} ... {j_b} $, where $a_1+...+a_t=a,\; i_1<...<i_t$ and $i_1 <j_1< ... <j_b.$

\subsection{Resolutions of hooks} We will use the the explicit finite projective resolution $P_{*}(a,b)$ of $ \Delta(\mathrm{h})$,
\[0 \rightarrow ... \rightarrow P_2(a,b) \xrightarrow{\theta_{2}(a,b)} P_1(a,b) \xrightarrow{\theta_{1}(a,b)} P_0(a,b)
\]
of \cite{Ma}, Theorem 1, which we now recall. For short we denote the tensor product $D_{a_1} \otimes ... \otimes D_{a_m}$ of divided powers by $D(a_1,...,a_m)$. Let
$P_i(a,b)=\sum D(a_1,...,a_{b+1-i})$  
where the sum ranges over all sequences $(a_1,...,a_{b+1-i})$ of positive integers of length $b+1-i$ such that $a_1+...+a_{b+1-i}=a+b$ and $a \le a_1 \le a+i$. The differential $\theta_{i}(a,b)$ is defined be sending $x_1 \otimes ... \otimes x_{b+1-i} \in D(a_1,...,a_{b+1-i})$ to
\[ \sum_{j=1}^{s} (-1)^{j+1}x_1 \otimes ...\otimes \triangle(x_j) \otimes ...  \otimes x_{b+1-i} \in D(a_1,...,u,v,...,a_{s}),
\]
where $ s=b+1-i $ and  $\triangle(x_j)$ is the image of $x_j$ under the two-fold diagonalization $D(a_j) \rightarrow \sum D(u,v)$, where the sum ranges of all positive integers $u,v$ such that $u+v=a_j$ and $D(a_1,...,u,v,...,a_{b+1-i})$ is a summand of $P_{i-1}(a,b)$ with $u$ located at position $j$. We denote by $ \triangle_{u,v} : D(a_j) \rightarrow D(u,v)$ the indicated component of the two-fold diagonalization $D(a_j) \rightarrow \sum D(u,v)$.

If $A,B$ are $S(n,r)$ - modules, we write $\mathrm{Hom}(A,B)$ and $\mathrm{Ext}^i(A,B)$ in place of $\mathrm{Hom}_{S(n,r)}(A,B)$ and $\mathrm{Ext}^i_{S(n,r)}(A,B)$ respectively.

We recall the recursions 
\begin{align*}
&P_0(a,b)=D(a) \otimes P_0(1,b-1), \\ &P_i(a,b)=P_{i-1}(a+1,b-1) \oplus D(a) \otimes P_i(1,b-1), i>0 \end{align*}
and that under these identifications we have the following. 

\begin{rem} If $M$ is a $S(n,r)$-module, the differential $\mathrm{Hom}(\theta_i(a,b), M)$ of the complex $\mathrm{Hom}(P_*(a,b),M)$ looks like
	\begin{center}
		\begin{tikzcd}
	\mathrm{Hom}(P_{i-2}(a+1,b-1),M) \arrow{r}  \arrow[d, phantom, "\oplus"]
		  & \mathrm{Hom}(P_{i-1}(a+1,b-1),M)\arrow[d, phantom, "\oplus"]\\
		\mathrm{Hom}(D(a) \otimes P_{i-1}(1,b-1),M)\arrow{r} \urar[shorten >= 25pt,shorten <= 25pt]{}
			  &\mathrm{Hom}(D(a) \otimes P_i(1,b-1),M)
		\end{tikzcd}
	\end{center} 
	where the top horizontal map is $\mathrm{Hom}(\theta_{i-1}(a+1,b-1),M)$, the bottom one is $-\mathrm{Hom}(1 \otimes  \theta _i(1,b-1),M)$ and the restriction of the diagonal one on the summand $\mathrm{Hom}(D(a,j,a_2,...,a_m),M)$  is $\mathrm{Hom}(\triangle_{a,j} \otimes 1 \otimes ... \otimes 1), M) .$
\end{rem}
For any $S(n,r)$-module $M$ and any sequence $a_1,...,a_m$ of non negative integers such that $a_1+...+a_m=r$ and $m \le n$,  we identify the $\mathbb{Z}$-module $$\mathrm{Hom}(D(a_1,...,a_m),M)$$ with the $(a_1,...,a_m)$ weight subspace of $M$ (with respect to the action of $\mathbb{Z}^n$) according to \cite{AB}, eqn. (11) on p. 178. We will use such identifications freely throughout this paper.

In particular, suppose  $M$ is a skew Weyl module for $S(n,r)$ (denoted be $K_{\lambda / \mu}(F)$ in \cite{ABW}). Using the $\mathbb{Z}$-basis of $M$  given by the semi-standard tableaux \cite{ABW}, Theorem II.3.16, we see that the $\mathbb{Z}$-module $\mathrm{Hom}(D(a_1,...,a_m),M)$ may be identified with the $\mathbb{Z}$-submodule of $M$ that has basis the semi-standard tableaux of $M$ that contain the entry $i$ exactly $a_i$ times, $i=1,...,m$. We call this the semi-standard basis of $\mathrm{Hom}(D(a_1,...,a_m),M)$. (Perhaps we should remark that what we have called semi-standard tableaux are called 'co-standard' in \cite{ABW}, Definition II.3.2: the entries in each row are weakly increasing from left to right and the entries in each column are strictly increasing from top to bottom.)

We record here a handy computational remark. For $M$ a skew Weyl module and $T \in \mathrm{Hom}(D(a_1,...,a_m),M)$ a semi-standard basis element, let $\phi_t (T)$, $1 \le t < m$, be the element of $\mathrm{Hom}(D(a_1,...,a_t+a_{t+1},...,a_m),M)$ obtained from $T$ by replacing each occurrence of $j>t$ by $j-1$. If $t\ge m$, let $\phi_t (T)=0$. By extending linearly, we obtain for each degree $i$ a map of $\mathbb{Z}$-modules $$\phi_t: \mathrm{Hom}(P_i(a,b),M) \rightarrow \mathrm{Hom}(P_{i+1}(a,b),M).$$ It is clear that only a finite number of these maps are nonzero. From the definition of the differential of $P_*(a,b)$, we obtain the following description for the differential of $\mathrm{Hom}(P_*(a,b),M)$.
\begin{rem}
	With the previous notation, $\mathrm{Hom}(\theta_i(a,b),M) = \sum_{i \ge 1}(-1)^{t-1}\phi_t.$
\end{rem}

Let $n \ge b+1$. Then $P_*(a,b)$ is a projective resolution of $\Delta(\mathrm{h})$. We claim that the $\mathbb{Z}$-module $\mathrm{Ext}^{i}(\Delta(\mathrm{h}),M)$, where $M$ is any skew Weyl $S(n,r)$-module, is isomorphic to the torsion submodule of the cokernel $E^i(\Delta(\mathrm{h}),M)$ of the map $\mathrm{Hom}(\theta_i(a,b),M)$. Indeed, by the argument of \cite{C}, bottom of p. 634 to the top of p. 635, we have \[E^i(\Delta(\mathrm{h}), M) \simeq \mathrm{Ext}^i(\Delta(\mathrm{h}), M) \oplus N, \]where $N$ is the image of the map $\mathrm{Hom}(\theta_i(a,b),M)$. (The argument given in loc. cit. is stated for $i=1$ but is valid for any $i \ge 1$.) As a submodule of a free $\mathbb{Z}$-module, $N$ is a free $\mathbb{Z}$-module. On the other hand, the $\mathbb{Z}$-module $ \mathrm{Ext}^i(\Delta(\mathrm{h}), M)$ is torsion for all $i \ge1$ by \cite{AB}, last paragraph of Section 8. Hence from the above isomorphism we obtain that the torsion submodule of $E^i(\Delta(\mathrm{h}),M)$ is isomorphic to $\mathrm{Ext}^i(\Delta(\mathrm{h}), M)$.

\subsection{The extensions $\mathrm{Ext}^i(\Delta(\mathrm{h}), D_{a+k}\otimes \wedge^{b-k})$}
We will use the following lemma several times. Let $$r_k = gcd\left(\tbinom{k+1}{1}, \dots ,\tbinom{k+1}{k}\right)$$ and note that $r_k=1$ unless $k+1=p^e$, $ p $ prime, in which case $r_k=p$.
\begin{lm}
	Suppose $n \ge {b+1}$ and $1 \le k < b$. Then $$\mathrm{Ext}^i(\Delta(\mathrm{h}),D_{a+k} \otimes \wedge ^{b-k})=\mathrm{Ext}^i(\wedge ^{k+1}, D_{k+1}).$$ In particular, $\mathrm{Ext}^1(\Delta (\mathrm{h}),D_{a+k} \otimes \wedge ^{b-k})= \mathbb{Z}_2$ and
	$\mathrm{Ext}^k(\Delta(\mathrm{h}),D_{a+k} \otimes \wedge^{b-k})= \mathbb{Z}_{r_k}.$
	\end{lm}
\begin{proof} A special case of the main result, Theorem 2,  of \cite{Ku1} yields $$\mathrm{Ext}^i(\Delta (a,1^b),D_{a+k} \otimes \wedge ^{b-k})=\mathrm{Ext}^i(D_{a-1} \otimes \wedge ^{k+1}, D_{a+k}).$$ Applying  contravariant duality \cite{Ja}, p. 209, and \cite{AB}, Theorem 7.7, we have $$\mathrm{Ext}^i(D_{a-1} \otimes \wedge ^{k+1}, D_{a+k})=\mathrm{Ext}^i(\wedge^{a+k}, D_{k+1} \otimes \wedge ^{a-1}).$$ Again by \cite{Ku1}, $\mathrm{Ext}^i(\wedge^{a+k}, D_{k+1} \otimes \wedge ^{a-1})=\mathrm{Ext}^i(\wedge^{k+1}, D_{k+1})$ and the first equality of the lemma follows.
	
	We have $\mathrm{Ext}^1(\wedge ^{k+1}, D_{k+1}) = \mathbb{Z}_2,$ by \cite{Ak}, Section 4, and $\mathrm{Ext}^k(\wedge ^{k+1}, D_{k+1}) = \mathbb{Z}_{r_k},$ according to \cite{Ma}, eqn. (6) p. 2207.
\end{proof}
\subsection{Summary of notation} For the reader's convenience we gather here some of the notation introduced in the previous subsections that will be used often.
\begin{itemize}[noitemsep]
		\item$\mathrm{h}=(a,1^b)$ and $\mathrm{h}(k)=(a+k,1^{b-k})$: hooks, where  $1\le k\le b$ and $r=a+b$, (subsection 2.1).
		\item ${i_1}^{(a_1)}  ...  {i_t}^{(a_t)} | {j_1} ... {j_b}$: the image in $\Delta(\mathrm{h})$, where $a=a_1+...+a_t$, of the element ${i_1}^{(a_1)}  ...  {i_t}^{(a_t)} \otimes {j_1} ... {j_b} \in D_a \otimes \wedge^b$ under the isomorphism  $cok(\theta_{a}) \simeq \Delta(\mathrm{h})$, (subsection 2.2).
		\item $P_{*}(a,b)$ and $\theta_{*}(a,b)$: projective resolution of $ \Delta(\mathrm{h})$ and the differential of $P_{*}(a,b)$ respectively, (subsection 2.3).
		\item $E^i(\Delta(\mathrm{h}),M)$: the cokernel of the map $\mathrm{Hom}(\theta_i(a,b),M)$, (subsection 2.3).
		\item $\phi_t$: a summand of the differential $\mathrm{Hom}(\theta_i(a,b),M)$ (defined before Remark 2.2).

\end{itemize}
\section{$\mathrm{Ext}^1(\Delta(\mathrm{h}),\Delta(\mathrm{h}(k)))$}
In this section we determine $\mathrm{Ext}^1(\Delta(\mathrm{h}),\Delta(\mathrm{h}(k)))$ for $k>1$. The case $k=1$ was computed in \cite{Ma}, Theorem 6.
\subsection{Matrices $e^{(1)}(a,b,M)$ and a generator} 
Let $M=D_{a+k} \otimes \wedge ^{b-k}$. For the semi-standard basis $B$ of the domain $\mathrm{Hom}(D(\mathrm{h}),M)$ $=\mathrm{Hom}(D(a,1,...,1),M)$ of the map $\mathrm{Hom}(\theta_1{(a,b)},M)$ we have $B=B_0 \cup B_1$, where $B_0$ (resp., $B_1$) consists of those elements of $B$ that have no (resp., exactly one) appearance of 1 in the $\wedge^{b-k}$ part. We consider the usual lexicographic ordering on $B$ and note that every element of $B_0$ is less than every element of $B_1$. We have $$|B|=\tbinom{b+1}{k+1}, |B_0|= \tbinom{b}{k}, |B_1|= \tbinom{b}{k+1}.$$

Likewise, for the semi-standard bases $B^1$, $B^2, ...,B^b $ of $\mathrm{Hom}(D(a+1,1,...,1),M)$, $\mathrm{Hom}(D(a,2,...,1),M)$, ... , $\mathrm{Hom}(D(a,1,...,2),M)$ respectively, we have $B^t=B^t_0 \cup B^t_1$, where $B_0^i$ (resp. $B_1^i$) consists of those elements of $B^i$ that have no (resp. exactly one) appearance of 1 in the $\wedge^{b-k}$ part, and thus \begin{equation*}B'=B^1_0 \cup B^1_1 \cup ... \cup B^b_0 \cup B^b_1\end{equation*} is a basis of the codomain of the map $\mathrm{Hom}(\theta_1{(a,b)},M)$. We order each set $B^i$ lexicographically and declare that every element of $B^i$ is less than every element of $B^{i+1}$. For each $i$ we have $$|B^i|=\tbinom{b}{k}, |B_0^i|= \tbinom{b-1}{k-1}, |B_1^i|= \tbinom{b-1}{k}.$$
Consider the matrix $e^{(1)}(a,b,M) \in M_{b\binom{b}{k} \times \binom{b+1}{k+1}}(\mathbb{Z})$ of the map $\mathrm{Hom}(\theta_1{(a,b)},M)$, with respect to the previous orderings, partitioned into $b$ row blocks according to $B'=B^1 \cup \cdots B^b $.

In the next Lemma, the missing entries of any matrix are assumed to be equal to 0.
\begin{lm} Let $p=\tbinom{b-1}{k-1}, q=\tbinom{b-1}{k}.$ The matrix  $e^{(1)}(a,b,M)$ has the following properties.\begin{enumerate} \item The first block is $\left( \begin{array}{c|c}  A(a,b;k) &B(a,b;k) \end{array}\right)$, where 	
 \begin{align*} &A(a,b;k)=diag(\underbrace{a+1,...,a+1}_{p},\underbrace{1,...,1}_{q}) \in M_{\binom{b}{k}}(\mathbb{Z}),  \\&B(a,b;k)= \left(
		\begin{array}{c|cc }
		&  &\\ \hline
		aI_{q} & &
		\end{array}
		\right) \in M_{\binom{b}{k} \times \binom{b}{k+1}}(\mathbb{Z}).\end{align*}
		\item The t-th block, $t>1$, is of the form $\left(
		\begin{array}{c|cc}
		C_t& \\ \hline
		&D_t
		\end{array}
		\right)$, where $C_t \in M_{p \times \binom{b}{k}}(\mathbb{Z})$, $D_t \in M_{q \times \binom{b}{k+1}}(\mathbb{Z})$.
		\item The sum of the elements in any row of the t-th block, where $t \ge 2$, is $(-1)^{t-1}2.$
	
	\item The last row is of the form $(0 \; ... \; 0 \; \pm 2)$.
	\end{enumerate}
\end{lm}
\begin{proof}
(1)  The set $B_0$ consists of all $T=1^{(a)}i_1...i_k\otimes j_1...j_{b-k} \in B$ such that 
\begin{gather*}i_1 <... < i_k, \; \; \; j_1 <... < j_{b-k} \\
\{i_1 ,..., i_k\} \cup \{j_1 ,... , j_{b-k}\}=\{2,...,b+1\} \\
\{i_1 ,..., i_k\} \cap \{j_1 ,... , j_{b-k}\}= \emptyset.
\end{gather*}
The definition of $\phi_1$ yields
\begin{equation*}
\phi_1 (T)=\begin{cases} (a+1)1^{(a+1)}i'_2...i'_k\otimes j'_1...j'_{b-k}, & \mbox{if}\;i_1=2 \\ \; \; \; \; \; \; \; \; \; \; \; 1^{(a)}i'_2...i'_k\otimes 1j'_2...j'_{b-k}, & \mbox {if}\; i_1 \ne 2 \end{cases},
\end{equation*}
where $i'=i-1$. From this it easily follows that the matrix of the restriction of $\phi_1$ on the subgroup of  $ \mathrm{Hom}(D(\mathrm{h}),M)$ spanned by $B_0$ is $A(a,b;k)$. 

The set $B_1$ consists of all $S=1^{(a-1)}i_1...i_{k+1}\otimes j_1...j_{b-k} \in B$ such that 
\begin{gather*}i_1 <... < i_{k+1}, \; \; \;  j_1 <... < j_{b-k}, \; \; \; j_1=1 \\
\{i_1 ,..., i_{k+1}\} \cup \{j_1 ,... , j_{b-k}\}=\{1,...,b+1\} \\
\{i_1 ,..., i_{k+1}\} \cap \{j_1 ,... , j_{b-k}\}= \emptyset.
\end{gather*}
Then 
\begin{equation*}
\phi_1 (S)=\begin{cases} a1^{(a-1)}i'_2...i'_k\otimes 1j'_2...j'_{b-k}, & \mbox{if}\;j_2 \ne 2 \\  0, & \mbox {if}\; j_2 = 2 \end{cases},
\end{equation*}
From this it easily follows that the matrix of the restriction of $\phi_1$ on the subgroup of  $ \mathrm{Hom}(D(\mathrm{h}),M)$ spanned by $B_1$ is $B(a,b;k)$. 

(2) If $T=1^{(a)}i_1...i_k\otimes j_1...j_{b-k} \in B_0$, then $j_1 \ge 2$, and thus for $t \ge 2$,  $\phi_t(T)$ is a multiple  of an element of $B^t$ that does not contain a 1 in the $\wedge^{b-k}$ part. Hence  $\phi_t(T) \in spanB^t_0$. If $S=1^{(a-1)}i_1...i_{k+1}\otimes j_1...j_{b-k} \in B_1$, then $j_1=1$ and thus $\phi_t(T) \in spanB^t_1$ for all $t$.

(3) Suppose $T' \in B^t$. Since $t>1$, $t$ appears exactly twice in $T'$.\\
\textit{Case 1.} Suppose $T'$ is of the form $T'=xt^{(2)}y\otimes z.$ Since each element of $ B $ has weight $ (a,1,...,1) $, from the definition of $ \phi_t $, it follows that there is a unique $T \in B$ such that $\phi_t(T)=cT', c \ne 0,$ namely $T=xt(t+1)y_1\otimes z_1$, where $y_1$ and $z_1$ are obtained from $y$ and $z$ respectively by replacing each $ i>t $ by $i+1$. We have $\phi_t(T)=2T'$.\\
\textit{Case 2.} Suppose $T'$ is of the form $T'=xty\otimes ztw.$ Since each element of $ B $ has weight $ (a,1,...,1) $, from the definition of $ \phi_t $, it follows that there are exactly two $T_1, T_2 \in B$ such that $\phi_t(T_i)=c_iT', c_i \ne 0,$ namely $$T_1=x(t+1)y_1\otimes ztw_1, \; T_2=xty_1\otimes x(t+1)w_1,$$ where $y_1$ and $w_1$ are obtained from $y$ and $w$ respectively by replacing each $ i>t $ by $i+1$. We have $\phi_t(T_1)=\phi_t(T_2)=T'$.

(4) This follows from  case 1 of (3) since the greatest element in $B^k$ is $T'=1^{(a-1)}(b-k+1)...(b-1)b^{(2)}\otimes 12...(b-k)$. 
\end{proof}
%\begin{lm}
%	If $X$ is one of $B_j, B_j^i$, let $\overline{X}=\sum_{T \in X}T$. Then %\begin{enumerate}
%	\item $\phi(\overline{B_0})=(a+1)\overline{B_0^1}+\overline{B_1^1}+2\sum_{t %\ge 2}(-1)^{t-1}\overline{B_0^t},$
%	\item %$\phi(a\overline{B_0}+\overline{B_1})=a(a+1)\overline{B_0^1}+2a\overline{B_1^1}%+2a\sum_{t \ge 2}(-1)^{t-1}\overline{B_0^t}+2\sum_{t \ge %2}(-1)^{t-1}\overline{B_1^t}.$
%	\end{enumerate}
%\end{lm}
%\begin{proof} The first equation follows from parts (1)-(3) of the previous %Lemma by adding the columns corresponding to ${B_0}$. Similarly, the second %equation follows from the first and the previous  Lemma.\end{proof}
Let $1\le k<b$ and $M=D_{a+k} \otimes \wedge ^{b-k}$. From Lemma 2.3 we have  $\mathrm{Ext}^1(\Delta(\mathrm{h}),M)$ $= \mathbb{Z}_2$. We will determine a generator of this $\mathrm{Ext}$ group.

We have mentioned that the torsion subgroup of the abelian group $E^1(\Delta(\mathrm{h}),M)$ with presentation matrix $e^{(1)}(a,b,M)$ is isomorphic to $\mathrm{Ext}^1(\Delta(\mathrm{h}),M)$. We denote by $\pi$ the natural projection $\pi:\mathrm{Hom}(P_{1}(a,b),M)\to E^1(\Delta(\mathrm{h}),M)$.
\begin{lm}
	Let $\mathrm{h}=(a,1^b)$, $1 \le k <b$ and $M=D_{a+k} \otimes \wedge ^{b-k}$. A cyclic generator of the abelian group $ \mathrm{Ext}^1(\Delta (\mathrm{h}),M)$ is $\pi(g_k)$, where
	\[g_k=  \tbinom{a+1}{2} \sum\limits_{T \in B_0^1} T + a \sum\limits_{T \in B_1^1} T + \sum\limits_{i=2}^{b}(-1)^{i-1} \Big( a\sum\limits_{T \in B_0^i}T + \sum\limits_{T \in B_1^i} T \Big).\] 
\end{lm}
\begin{proof}
Let $E_i$ be the i-th column of $e^{(1)}(a,b,M)$ and let $p=\binom{b-1}{k-1}, q=\binom{b-1}{k}$. Consider the $e^{(1)}(a,b,M)$ partitioned into $b$ blocks each consisting of $p+q=\binom{b}{k}$ consecutive rows.  From Lemma 3.1 it follows that 
\begin{align} \nonumber & a\big(E_1+...+E_{\binom{b}{k}}\big)+E_{\binom{b}{k}+1}+...+E_{\binom{b+1}{k+1}}= \\  \nonumber &\Big( \underbrace{a(a+1),...,a(a+1)}_{p}, 
\underbrace{2a,...,2a}_{q}, 
\underbrace{-2a,...,-2a}_{p}, 
\underbrace{-2,...,-2}_{q},..., \\ \label{lin_com}
&\underbrace{(-1)^{b-1}2a,...,(-1)^{b-1}2a}_{p}, 
\underbrace{(-1)^{b-1}2,...,(-1)^{b-1}2}_{q} \Big)^t. 
\end{align}
Hence in the cokernel $E^1(\Delta(\mathrm{h}),M)$ of the differential $\mathrm{Hom}(\theta_1(a,b),M)$ we have $2\pi(g_k)=0$. This shows that $\pi(g_k) \in \mathrm{Ext}^1(\Delta (\mathrm{h}),M)$. We have $\pi(g_k)\ne 0$, since otherwise the integer matrix-column  $\frac{1}{2} X$, where $X$ is the right hand side of (\ref{lin_com}), would be a $\mathbb{Z}$ - linear combination of columns of  $e^{(1)}(a,b,M)$. This is not possible because the last entry of  $\frac{1}{2} X$ is $\pm1 $ while all entries of the last row of $e^{(1)}(a,b,M)$, $k>0$, are even according to Lemma 3.1 (4). Since 
$ \mathrm{Ext}^1(\Delta(\mathrm{h}),M)= \mathbb{Z}_2$, it follows that $\pi(g_k)$ generates $\mathrm{Ext}^1(\Delta(\mathrm{h}),M)$.\end{proof}

\subsection{Proof for $\mathrm{Ext}^1(\Delta(\mathrm{h}),\Delta (\mathrm{h}(k)))$.} We determine $\mathrm{Ext}^1(\Delta(\mathrm{h}),\Delta(\mathrm{h}(k))), k \ge 2$ in this subsection. The main computation is done in the next two lemmas the first of which takes care of the case $k=2$. 

We will use the following notation. If $f:M \to N$ is a map of $S(n,r)$-modules, we have in the usual way various induced maps \begin{align*}\mathrm{Hom}(P_i(a,b),M) &\to \mathrm{Hom}(P_{i}(a,b),N), \\E^i(\Delta(\mathrm{h}),M) &\to E^i(\Delta(\mathrm{h}),N), \\ \mathrm{Ext}^i(\Delta(\mathrm{h}),M) &\to \mathrm{Ext}^i(\Delta(\mathrm{h}),N)\end{align*} of abelian groups which will all will be denoted by $f^*$. For an integer $m$ let $\epsilon_m$ be the remainder of the division of $m$ by 2.

In the statement of the next Lemma, we recall that $\mathrm{\mathrm{Ext}}^1(\Delta(\mathrm{h}),D_{a+1} \otimes \wedge ^{b-1}) = \mathbb{Z}_2$ according to Lemma 2.3. Also from \cite{Ma}, Theorem 6, we have that $\mathrm{Ext}^1(\Delta(\mathrm{h}),\Delta(\mathrm{h}(1)))$ is a cyclic group.

\begin{lm}
	Let $\mathrm{h}=(a,1^b)$ and $\mathrm{h}(1)=(a+1,1^{b-1}),  b \ge 2$. The map $$ \mathrm{\mathrm{Ext}}^1(\Delta(\mathrm{h}),D_{a+1} \otimes \wedge ^{b-1}) \xrightarrow{\pi_0^{*}} \mathrm{Ext}^1(\Delta(\mathrm{h}),\Delta(\mathrm{h}(1)))$$
	induced by $D_{a+1} \otimes \wedge ^{b-1}  \xrightarrow{\pi_0} \Delta (\mathrm{h}(1))$ is multiplication by the integer $ \frac{(a+ \epsilon_b-1)(a+b)}{2} $. 
\end{lm}

\begin{proof} The matrix of the map $\mathrm{Hom}(P_0(a,b),\Delta(\mathrm{h}(1))) \rightarrow \mathrm{Hom}(P_1(a,b),\Delta(\mathrm{h}(1)))$ with respect to the lexicographic order of semi-standard bases is the following $b \times b$ matrix according to \cite{Ma}, p. 2211,
\[
\left(
\begin{array}{cccccc}
a+1 &-1 &1 & \dots & &(-1)^{b-1} \\
-1  &-1 &0 & \dots & &0 \\
0   &1  &1 &  \dots & &0  \\
\vdots &\vdots &\vdots & & &\vdots \\
0 &0 &0 &\dots &(-1)^{b-1} &(-1)^{b-1} 
\end{array}
\right).
\]
%Therefore $E^1(\Delta(\mathrm{h}),\Delta(\mathrm{h}(1)))=$ $\mathrm{Ext}^1(\Delta(\mathrm{h}),\Delta (\mathrm{h}(1)))  $.
Let $$S_1=1^{(a+1)}|23...b, S_2=1^{(a)}2|23...b,...,S_b=1^{(a)}b|23...b$$ be the semi-standard basis of 
$\mathrm{Hom}(P_1(a,b),\Delta(\mathrm{h}(1)))$. From the above matrix it follows that for each $i$ there is an integer $m_i$ such that $\pi(S_i) = m_i\pi(S_1)$, and hence a cyclic generator of $E^1(\Delta(\mathrm{h}),\Delta(\mathrm{h}(1)))$ is $\pi(S_1)$. Moreover by adding the even numbered columns $2,4,...,2[\frac{b}{2}]$ we have the relation $$\sum\limits_{i=2}^{b}(-1)^{i-1}\pi(S_i) =[\frac{b}{2}]\pi(S_1)$$
where $[\frac{b}{2}]$ is the largest integer less than or equal to $\frac{b}{2}$.

Using the notation established at the beginning of subsection 3.1, we have
$B_0^1=\{T_1\}$ and $B_1^1=\{T_2,..,T_b\}$, where $T_1=1^{(a+1)}\otimes 23...b$ and $T_j=1^{(a)}j\otimes 12...\widehat{j}...b, j=2,...,b$, and it is understood that $\widehat{j}$ means that $j$ is omitted. Now $\pi_0^*(T_1)=S_1$. Also  $$\pi_0^*(T_j)=1^{(a)}j|12...\widehat{j}...b = -1^{(a+1)}|j2...\hat{j}...b=(-1)^{j-1}S_1, $$ where in the second equality the  straightening law was used. Thus 
\[ \pi_0^* \Big( \tbinom{a+1}{2} \overline{B_0^1} + a \overline{B_1^1} \Big)  = \Big(\tbinom{a+1}{2} + a(\epsilon_b-1) \Big)S_1,
\]
where $ \overline{X} = \sum_{T \in X}T$, if  $X$ is one of the sets $B_0^i, B_1^i$. A similar computation for $j=2,..,b$ yields 
\[ \pi_0^* \Big( \overline{B_0^j}  +  \overline{ B_1^j} \Big) = \big(a +  \epsilon_b-1 \big)S_j, 
\]
and therefore in $E^1(\Delta(\mathrm{h}),\Delta(\mathrm{h}(1)))$ we have
\begin{align*} \nonumber 
\pi_0^*(\pi(g_1)) & = \Big(\tbinom{a+1}{2} + a(\epsilon_b-1) \Big)\pi(S_1) +\big( a+ \epsilon_b -1 \big)  \sum\limits_{i=2}^{b} (-1)^{i-1} \pi(S_i) \\
& =\Big(\tbinom{a+1}{2} + a(\epsilon_b-1)  + \big( a+ \epsilon_b -1 \big)[\frac{b}{2}] \Big) \pi(S_1).
\end{align*}
It is easy to verify that $\tbinom{a+1}{2} + a(\epsilon_b -1)  + \big( a+ \epsilon_b -1 \big)[\frac{b}{2}] = \frac{(a+ \epsilon_b -1)(a+b)}{2} $. We have shown that the map $ \mathrm{Ext}^1(\Delta(\mathrm{h}),D_{a+1} \otimes \wedge ^{b-1}) \xrightarrow{\pi_0^{*}} \mathrm{Ext}^1(\Delta (\mathrm{h}),\Delta (\mathrm{h}(1)))$ is multiplication by this integer. Since $g_1$ is a generator of $\mathrm{Ext}^1(\Delta(\mathrm{h}),D_{a+1} \otimes \wedge ^{b-1})$ (Lemma 3.2) which is the torsion submodule of $E^1(\Delta(\mathrm{h}),D_{a+1} \otimes \wedge ^{b-1})$, the result follows.\end{proof}
Since multiplication in the divided power algebra is commutative, we will often denote a semi-standard basis element of the form  $1^{(a)}i_1...i_st^{(2)}i_{s+1}...i_{k-2}\otimes j_1...j_{b-k}$ by $1^{(a)}t^{(2)}i_1...i_{k-2}\otimes j_1...j_{b-k}$ and likewise for $1^{(a)}i_1...i_sti_{s+1}...i_{k-1}\otimes j_1...t...j_{b-k-1}$.

The $\mathrm{Ext}$ groups appearing in the next Lemma are both equal to $\mathbb{Z}_2$ by Lemma 2.3. We want to identify a particular map between these.

\begin{lm}
	Let $\mathrm{h}=(a,1^b)$ and $ 1<k<b$. The map $$ \mathrm{Ext}^1(\Delta (\mathrm{h}),D_{a+k} \otimes \wedge ^{b-k}) \xrightarrow{\theta^{*}} \mathrm{Ext}^1(\Delta(\mathrm{h}),D_{a+k-1} \otimes \wedge ^{b-k+1})$$
	induced by $D_{a+k} \otimes \wedge ^{b-k}  \xrightarrow{\theta} D_{a+k-1} \otimes \wedge ^{b-k+1}$ is multiplication by $a+ \epsilon_{b-k+1}-1 $.
\end{lm}

\begin{proof} Consider the semi-standard basis $B'=B^1_0 \cup B^1_1 \cup ... \cup B^b_0 \cup B^b_1$, with the notation as in the beginning of subsection 3.1, of the codomain of the map $\mathrm{Hom}(\theta_1{(a,b)},D_{a+k} \otimes \wedge ^{b-k})$. Likewise we denote by $C'=C^1_0 \cup C^1_1 \cup ... \cup C^b_0 \cup C^b_1$ the semi-standard basis of the codomain of the map $\mathrm{Hom}(\theta_1{(a,b)},D_{a+k-1} \otimes \wedge ^{b-k+1})$. If $ X $ is any one of the sets $B_0^t, B_1^t, C_0^t, C_1^t$, we let $ \overline{X} = \sum_{T \in X}T .$ 
	%According to our usual abuse of notation, we denote again by $ %\overline{X}$ its image in $ \mathrm{Ext}^1(\Delta (\mathrm{h}),D_{a+k} \otimes \wedge ^{b-k})$ %(for $X=B_j^i$) or in $ \mathrm{Ext}^1(\Delta (\mathrm{h}),D_{a+k-1} \otimes \wedge ^{b-k+1})$ %(for $X=C_j^i$)

We claim that for each $t=1,...,b,$
 \begin{align}
&\theta^{*}(\overline{B^t_0}) =\epsilon_{b-k+1} \overline{C^t_0}+\overline{C^t_1} \label{claim1}, \\
&\theta^{*}(\overline{B^t_1}) =-\epsilon_{b-k} \overline{C^t_1}\label{claim2}.
\end{align}

Indeed, let $t>1$. We note that $B_0^t$ consists of all $1^{(a)}t^{(2)}i_1...i_{k-2}\otimes j_1...j_{b-k}$ such that
 \begin{gather*}i_1<...<i_{k-2}, \; \; j_1<...<j_{b-k},\\ \{ i_1,...,i_{k-2} \} \cap  \{ j_1,...,j_{b-k} \} = \emptyset, \; \{ i_1,...,i_{k-2} \} \cup  \{ j_1,...,j_{b-k} \} = \{2,...,\widehat{t},...,b\} \end{gather*}
and of all
$1^{(a)}ti_1...i_{k-1}\otimes j_1...t...j_{b-k-1}$ such that
 \begin{gather*}i_1<...<i_{k-1}, \; \; j_1<...<t<...<j_{b-k-1},\\ \{ i_1,...,i_{k-1} \} \cap  \{ j_1,...,j_{b-k-1} \} = \emptyset, \; \{ i_1,...,i_{k-1} \} \cup  \{ j_1,...,j_{b-k-1} \} = \{2,...,\widehat{t},...,b\} \end{gather*}

%Likewise, $B_1^t$  consists of the elements %$$1^{(a)}i_1...i_st^{(2)}i_{s+1}...i_{k-2}\otimes 1j_1...j_{b-k-1}$$ such that
%\begin{align*}&2 \le i_1<...<i_{s}<t<i_{s+1}<...<i_{k-2} \le b,\\&2 \le %j_1<...<j_{b-k-1} \le b,\\ &\{ i_1,...,t,...,i_{k-2} \} \cap  \{ j_1,...,j_{b-k-1} \} %= \emptyset,  \end{align*}
%and of the elements
%$$1^{(a)}i_1...i_sti_{s+1}...i_{k-1}\otimes 1j_1...t...j_{b-k-2}$$ such that
%\begin{align*}&2 \le i_1<...<i_{s}<t<i_{s+1}<...<i_{k-1} \le b,\\&2 \le %j_1<...<t<...<j_{b-k-2} \le b,\\ &\{ i_1,...,t,...,i_{k-1} \} \cap  \{ %j_1,...,j_{b-k-2} \} = \emptyset.  \end{align*}

%Replacing $k$ by $k-1$ in the above descriptions yields the elements of the bases %$C_0^t$ and $C_1^t$ respectively.
The definition of $\theta^*$ on the above elements yields 
 \begin{align}
 \nonumber\theta^{*}(1^{(a)}t^{(2)}i_1...i_{k-2}\otimes j_1...j_{b-k}) 
 = &1^{(a-1)}t^{(2)}i_{1}...i_{k-2}\otimes1j_1...j_{b-k}\\ \nonumber
 &+1^{(a)}ti_1...i_{k-2}\otimes tj_1...j_{b-k}\\
 &+\sum_{u=1}^{k-2}1^{(a)}t^{(2)}i_1...\widehat{i_u}...i_{k-2} \otimes i_uj_1...j_{b-k})\label{diaf1},\end{align}
 \begin{align}
 \nonumber\theta^{*} (1^{(a)}ti_1...i_{k-1}\otimes j_1...t...j_{b-k-1}) 
 = &1^{(a-1)}ti_1...i_{k-1}\otimes1j_1...t...j_{b-k-1}\\
 &+\sum_{u=1}^{k-1}1^{(a)}ti_1...\widehat{i_u}...i_{k-1} \otimes i_uj_1...t...j_{b-k-1})\label{diaf2},\end{align}
where $\widehat{i_u}$ means that $i_u$ is omitted. We note that each term in the right hand side of equations (3.4) and (3.5) is of the form $\pm S$, where $S \in C^t$. Moreover, the terms in the right hand side of (3.4) are distinct and those in the right hand side of (3.5) are distinct. Now let $ S \in C^t=C^t_0 \cup C^t_1. $ \\
(1) Let $S \in C_0^t$.
	\begin{enumerate}
	\item[(a)] Suppose $S=1^{(a)}tu_1...u_{k-2} \otimes v_1...t...v_{b-k}$. From (3.4) and (3.5), it follows that the elements $T_i \in B^t$ such that $S$ appears with nonzero coefficient in $\theta^*(T_i)$ are
	\begin{align*}
	&&T_0&=1^{(a)}t^{(2)}u_1...u_{k-2}\otimes v_1...v_{b-k}, \\
	&&T_i&=1^{(a)}tu_1...u_{k-2}v_i\otimes v_1...\widehat{v_i}...t...v_{b-k}, \; i=1,...,b-k.
	\end{align*}
	Moreover by straightening the $\wedge^{b-k+1}$ part, the coefficient of $S$ in $\theta^*(T_0)$ is $ (-1)^s $, $s=\#  \{i:v_i<t\}$, and the coefficient of $S$ in $\theta^*(T_i)$ is $$\begin{cases} (-1)^{i-1}, & \mbox{if}\; \mbox{$i\le s$} \\ (-1)^i, &  \mbox{$i\ge s+1$}. \end{cases}$$
	Therefore, by summing over $B_0^t$ be see that the coefficient of $S$ in $\theta^{*}(\overline{B^t_0})$ is $\sum_{i=1}^{b-k+1}(-1)^{i-1}=\epsilon_{b-k+1}.$
	\item[(b)] Suppose $S=1^{(a)}t^{(2)}u_1...u_{k-3}\otimes v_1...v_{b-k+1}$. From (3.4) and (3.5) it follows that the elements $T_i \in B^t$ such that $S$ appears with nonzero coefficient in $\theta^*(T_i)$ are
	\begin{align*}
	&&T_i=1^{(a)}t^{(2)}u_1...u_{k-3}v_i\otimes v_1...\widehat{v_i}...t...v_{b-k}, \; i=1,...,b-k.
	\end{align*}
	Moreover by straightening the $\wedge^{b-k+1}$ part, the coefficient of $S$ in $\theta^*(T_0)$ is $(-1)^{i-1}$. Thus summing over $B_0^t$, the coefficient of $S$ in $\theta^*(\overline{B_0^t})$ is $\sum_{i=1}^{b-k+1}(-1)^{i-1}=\epsilon_{b-k+1}$.\end{enumerate}
	(2) Let $S \in C_1^t$.
	\begin{enumerate}
	\item[(a)]	Suppose $S=1^{(a-1)}t^{(2)}u_1...u_{k-2}\otimes 1v_1...v_{b-k}$. From (3.4) and (3.5), it follows that there is a unique $T \in B_0^t$ such that $S$ appears with nonzero coefficient in $\theta^*(T)$, namely $T=1^{(a)}t^{(2)}u_1...u_{k-2}\otimes v_1...v_{b-k}$, and the coefficient is 1.
	
	\item[(b)]	Suppose $S=1^{(a-1)}tu_1...u_{k-1}\otimes 1v_1...t...v_{b-k-1}$. From (3.4) and (3.5), it follows that there is a unique $T \in B_0^t$ such that $S$ appears with nonzero coefficient in $\theta^*(T)$, namely $T=1^{(a)}tu_1...u_{k-1}\otimes v_1...t...v_{b-k-1}$, and the coefficient is 1.
	\end{enumerate}
	
From the cases (1) and (2), equation (3.2) follows for $ t>1. $ The proof of (3.3), $t>1$, is similar (and a bit shorter) and omitted. Finally, the proof of (3.2) and (3.3) for $ t=1 $ is similar (and a bit simpler) and omitted.

We now prove the statement of the Lemma.

\noindent$\it{Case \; 1.}$ Suppose $b-k+1$ is even. By substituting (3.2) and (3.3) we obtain
\[
\theta^{*}(g_{k})= \tbinom{a}{2} \overline{C^1_1} + (a-1)\sum_{i=2}^{b}(-1)^{i-1}\overline{C_0^i}
\]
and thus
\[ 
\theta^{*}(g_k) -(a-1)g_{k-1}= -
\tbinom{a}{2} \Big( (a+1) \overline{C^1_0} +\overline{C^1_1} + 2\sum_{i=2}^{b} (-1)^{i-1}\overline{C^i_0} \Big). 
\]
But $ \pi\big((a+1) \overline{C^1_0} +\overline{C^1_1} + 2\sum_{i=2}^{b} (-1)^{i-1}\overline{C^i_0}\big)=0$ in $E^1(\Delta(\mathrm{h}),D_{a+k-1} \otimes \wedge ^{b-k+1})$ because this is the relation coming from adding the first $\tbinom{b}{k-1}$ columns of the matrix $e^{(1)}(a,b,D_{a+k-1}\otimes\wedge^{b-k+1})$ according to Lemma 3.1 (1)-(3). Thus in this case the map $\theta^{*}$ is multiplication by $a-1$.

\noindent$\it{Case \; 2.}$ Suppose $b-k+1$ is odd. 
By substituting (3.2) and (3.3) we obtain
\[
\theta^{*}(g_k)= \tbinom{a+1}{2} \big( \overline{C^1_0} + \overline{C^1_1} \big) + a  \sum_{i=2}^{b} (-1)^{i-1}(\overline{C^i_0}+\overline{C^i_1})
\]
and using this we have 
\[
\theta^{*}(g_k) - ag_{k-1}= 
-\tbinom{a}{2} \Big( (a+1) \overline{C^1_0} +\overline{C^1_1} + 2\sum_{i=2}^{b} (-1)^{i-1}\overline{C^i_0} \Big).
\]
Thus in this case the map $\theta^{*}$ is multiplication by $a$. \end{proof}
\begin{theo}
	Let $\mathrm{h}=(a,1^b)$ and $\mathrm{h}(k)=(a+k,1^{b-k})$, where $ 2 \le k \le b$. If $n \ge b+1$, then
	$\mathrm{Ext}^1(\Delta(\mathrm{h}), \Delta(\mathrm{h}(k)) = 0$ unless $a+b+k$ is odd in which case $\mathrm{Ext}^1(\Delta(\mathrm{h}), \Delta(\mathrm{h}(k)) = \mathbb{Z}_2$.
\end{theo}
\begin{proof}
Applying $\mathrm{Hom}(\Delta(\mathrm{h}),-)$ to the short exact sequence sequence (2.1) for $\mathrm{h}(k-1)$ in place of $\mathrm{h}$ yields the exact sequence 
\begin{align*}
0 &\rightarrow \mathrm{Ext}^1(\Delta(\mathrm{h}),\Delta(\mathrm{h}(k)) \xrightarrow{i*} \mathrm{Ext}^1(\Delta(\mathrm{h}),D_{a+k-1} \otimes \wedge^{b-k+1}) \\ &\xrightarrow{\pi_0^*} \mathrm{Ext}^1(\Delta(\mathrm{h}),\Delta(\mathrm{h}(k-1)))
\end{align*}
because $\mathrm{Hom}(\Delta(\mathrm{h}),\Delta(\mathrm{h}(k-1)))=0$ as $ \mathbb{Q} \otimes \Delta(\mathrm{h})$ and $ \mathbb{Q} \otimes \Delta(\mathrm{h}(k-1))$ are distinct irreducible representations of $GL_n( \mathbb{Q} )$. 

First let $k=2$. From the above exact sequence, Lemma 2.3 and \cite{Ma}, Theorem 6, it follows that $\mathrm{Ext}^1(\Delta(\mathrm{h}),\Delta(\mathrm{h}(2)))$ is the kernel of the map $\mathbb{Z}_2 \rightarrow \mathbb{Z}_{a+b}$ which is multiplication by the integer  $ \frac{(a+ \epsilon_b-1)(a+b)}{2}$ according to Lemma 3.3. Thus the result follows.

Suppose $k \ge 3$. Then $\mathrm{Ext}^1(\Delta(\mathrm{h}),\Delta(\mathrm{h}(k-1)))$ injects in $\mathrm{Ext}^1(\Delta(\mathrm{h}),D_{a+k-2} \otimes \wedge^{b-k+2})$ and thus $\mathrm{Ext}^1(\Delta(\mathrm{h}),\Delta(\mathrm{h}(k))$ is the kernel of the composite map
\[
\psi : \mathrm{Ext}^1(\Delta(\mathrm{h}),D_{a+k-1} \otimes \wedge^{b-k+1}) \rightarrow \mathrm{Ext}^1(\Delta(\mathrm{h}),D_{a+k-2} \otimes \wedge^{b-k+2}).
\]
This map is induced by $D_{a+k-1} \otimes \wedge ^{b-k+1}  \xrightarrow{\theta} D_{a+k-2} \otimes \wedge ^{b-k+2}$. According to Lemma 3.4,  $\psi : \mathbb{Z}_2 \rightarrow \mathbb{Z}_2$  is by multiplication by $a+ \epsilon_{b-k}-1 $. Hence the result follows. \end{proof}

\section{$\mathrm{Ext}^k(\Delta(\mathrm{h}),\Delta(\mathrm{h}(k)))$}
Let $\mathrm{h}=(a,1^b)$ and $\mathrm{h}(k)=(a+k,1^{b-k})$, where we assume throughout this section that $1 \le k \le b$. We will prove the following result.
\begin{theo}If $n \ge b+1$, then $\mathrm{Ext}^k(\Delta(\mathrm{h}), \Delta(\mathrm{h}(k)))= \mathbb{Z}_{d_k}$, where $ d_k=gcd (\tbinom{a+b}{1},\tbinom{a+b}{2}, \dots ,\tbinom{a+b}{k} )$.
\end{theo}
	This result is known in the special cases $a=1$, $b=k$ \cite{Ak}, Section 4, and 
		any $a$, $b=k$ \cite{Ma}, eqn. (6) p. 2207. 
		
		According to the following Remark, the above $\mathrm{Ext}$ group is the highest possible nonzero $\mathrm{Ext}$ group between the hooks $\mathrm{h}$ and $\mathrm{h}(k)$. We thank both H. H Andersen and the referee for pointing out an error in a previous version of this paper concerning the proof of the Remark and for suggesting the proof that follows.
		
		\begin{remark}Let $n \ge b+1$. If $i>k$, then $\mathrm{Ext}^i(\Delta(\mathrm{h}), \Delta(\mathrm{h}(k)))=0.$\end{remark} 
		\begin{proof}
			We use induction on $k$. For $k=0$ the result follows from the general fact that no Weyl module has non trivial self extension, see [10], B.4. Remark. Applying $\mathrm{Hom}(\Delta (\mathrm{h}),-)$ to the short exact sequence
			\[0\rightarrow\Delta(\mathrm{h}(k+1))\rightarrow \Delta(a+k)\otimes  \Lambda(b-k) \rightarrow \Delta(\mathrm{h}(k))\rightarrow 0\] we obtain the exact sequence
			
			\begin{align*} \mathrm{Ext}^i(\Delta(\mathrm{h}), \Delta(\mathrm{h}(k)))&\rightarrow \mathrm{Ext}^{i+1}(\Delta(\mathrm{h}),\Delta(h(k+1))) \\&\rightarrow \mathrm{Ext}^{i+1}(\Delta(\mathrm{h}),\Delta(a+k)\otimes  \Lambda^{b-k}). \end{align*}
			The term on the left is zero by induction. By Lemma 2.3, the term on the right is 
			$$\mathrm{Ext}^{i+1}(\Delta(\mathrm{h}),\Delta(a+k)\otimes  \Lambda^{b-k})=\mathrm{Ext}^{i+1}(\Lambda^{k+1},\Delta(k+1))$$
			which is zero by induction since $i+1>k$. Hence the middle term is zero.
		\end{proof}

Recall the following notation from Section 2.3. $E^i(\Delta(\mathrm{h}),M)$ is the cokernel of the differential $\mathrm{Hom}(\theta_i(a,b),M)$ of the complex $\mathrm{Hom}(P_*(a,b),M)$, where $M$ is a skew Weyl module. The torsion part of this abelian group is isomorphic to $\mathrm{Ext}^i(\Delta(\mathrm{h}),M)$. Let $\pi$ the natural projection $\pi:\mathrm{Hom}(P_{i}(a,b),M)\to E^i(\Delta(\mathrm{h}),M)$
and $e^{(i)}(a,b,M)$ the matrix of the map $\mathrm{Hom}(\theta_i(a,b),M)$
with respect to orderings to be specified.

We order lexicographically the semi-standard basis of $\mathrm{Hom}(D(a_1,...,a_m),\Delta(\mathrm{h}(k))).$ Now if  $(a_1,...,a_m)$ is greater than   $(b_1,...,b_{m'})$ , where $m,m' \le n$, in the usual lexicographic ordering of sequences, we declare that each element of the semi-standard basis of $\mathrm{Hom}(D(a_1,...,a_m),\Delta(\mathrm{h}(k)))$ is less than each element of the semi-standard basis of $\mathrm{Hom}(D(b_1,...,b_{m'}),\Delta(\mathrm{h}(k)))$. With respect to the above orderings, Remark 2.1 yields the following, where the missing entries in the bottom left part of the matrix are equal to zero.
\begin{lm}
	For $i>1$ and $b>1,$ $
	e^{(i)}(a,b,\Delta(\mathrm{h}(k)))=
	\left(
	\begin{array}{c|c}
	A &*\\
	\hline \;
	& *
	\end{array}
	\right),$
	where $A=e^{(i-1)}(a+1,b-1,\Delta(\mathrm{h}(k)))$.
\end{lm}

\subsection{A generator of $\mathrm{Ext}^k(\Delta(\mathrm{h}),D_{a+k}\otimes\wedge^{b-k})$}
In this subsection we will identify a generator of $\mathrm{Ext}^k(\Delta(\mathrm{h}),D_{a+k} \otimes \wedge^{b-k})$. We assume throughout that $1 \le k \le b.$

	Let $\Gamma_k \in \mathrm{Hom}(P_k(a,b),D_{a+k}\otimes\wedge^{b-k})$,
	\begin{equation}\Gamma_k = \tbinom{a+k}{k+1} \tr_{1} - \tr_{2} + \cdots + {(-1)}^{b-k} \tr_{q},\end{equation} where $q=b-k+1$ and \begin{align*}\tr_{1} &= 1^{(a+k)} \otimes 2 \dots q \in \mathrm{Hom}(D(a+k,1,...,1),D_{a+k}\otimes \wedge^{b-k})\\
\tr_{i} &= 1^{(a-1)} i^{(k+1)} \otimes 1\dots\widehat{i} \dots q \\ 
&\,\,\,\,\,\,\in \mathrm{Hom}(D(a,1,\dots,k+1,\dots,1), D_{a+k} \otimes \wedge^{b-k}),\end{align*}
$i = 2, \dots,q,$ where $k+1$ is located at the i-th position. Consider the natural projection $\pi: \mathrm{Hom}(P_k(a,b), D_{a+k}\otimes\wedge^{b-k}) \to E^k(\Delta(\mathrm{h}),D_{a+k}\otimes\wedge^{b-k}).$

\begin{lm} If $k+1=p^e$, $p$ prime, then $\pi(\Gamma_k) \ne 0.$ 
\end{lm}
\begin{proof}Suppose $\Gamma_k$ is equal to a linear combination of the columns of $A$, where $A=e^{(k)}(a,b,D_{a+k}\otimes \wedge^{b-k})$. Then the coefficient ${(-1)}^{q-1}$ of $\tr_{q}$, $q=b-k+1$, is a linear combination of the entries of the last row of $A$, as $\tr_{q}$ is the last basis element in $\mathrm{Hom}(P_k(a,b),D_{a+k}\otimes \wedge^{b-k})$ with respect to the lexicographic order. We claim that the set of nonzero elements in the last row of $A$ is $\{(-1)^{q-1}\tbinom{k+1}{1},\dots ,(-1)^{q-1}\tbinom{k+1}{k}\}.$

Indeed, let $i\in \{1,\dots ,q\}$, $\lambda = (\lambda_1,\dots ,\lambda_{q+1})$, where $\lambda_1 \in \{0, \dots ,k-1\}$, $\lambda_j \geq 1$ for $j \in\{2,\dots, q+1\}$, $\sum_{j=1}^{q+1}\lambda_j =b$, and consider a semi-standard basis element $T \in \mathrm{Hom}(D(a+\lambda_1, \lambda_2, \dots ,\lambda_{q+1}),D_{a+k}\otimes \wedge^{b-k})$, such that $$\phi_i(T)=c\tr_{q}, \; c\in \mathbb{Z}-\{0\}.$$ 
By Remark 2.2, the left hand side has weight $$\left(1^{a+\lambda_1},2^{\lambda_2},\dots,i^{\lambda_i + \lambda_{i+1}}, \dots ,q^{\lambda_{q+1}}\right),$$
 which must be equal to the weight $\left(1^a,2,\dots, q-1,  q^{k+1}\right)$ of the right hand side. So $\lambda_1 =0$ and if $1<i\leq q-1$ then $\lambda_i + \lambda_{i+1}=1$ which contradicts the hypothesis $\lambda_i \geq 1$ for $i \in \{2,\dots ,q+1\}$. This implies that $i=q$, $\lambda_1=0$, $\lambda_j=1$ for $j \in \{2,\dots, q-1\}$ and $\lambda_{q} + \lambda_{q+1}=k+1$. Hence $T=1^{(a-1)}{q}^{\lambda_{q}} {(q+1)}^{\lambda_{q+1}} \otimes1 \dots (q-1)$, where $\lambda_{q} + \lambda_{q+1} =k+1$. For such a $T$ we have $(-1)^{q-1}\phi_{q}(T)=(-1)^{q-1}\tbinom{k+1}{\lambda_q}\tr_{q},$
which proves the claim.

It follows that $gcd\left(\tbinom{k+1}{1},\dots,\tbinom{k+1}{k}\right)=1$ contradicting the assumption $k+1=p^e$, $p$ prime. Hence $\pi(\Gamma_k)\neq 0$. \end{proof}

Let $q=b-k+1$. Define $T_{1,j} \in \mathrm{Hom}(D(a+k-1, 1,\dots ,1), D_{a+k}\otimes \wedge^{b-k})$, $j=2,\dots,q+1$, 
$$T_{1,j}=1^{(a+k-1)}j \otimes 2\dots \widehat{j} \dots (q+1), $$
and $T_{i,j} \in \mathrm{Hom}(D(a,1,\dots,k,\dots,1), D_{a+k}\otimes \wedge^{b-k})$, where $k$ is at the i th-position, $i=2,\dots,q$, $j=i+1,\dots,q+1,$ $$T_{i,j}=1^{(a-1)}i^{(k)}j \otimes1\dots\widehat{i}\dots\widehat{j} \dots (q+1). $$
Let $$A=\tbinom{a+k-1}{k} \sum_{j=2}^{q+1}(-1)^jT_{1,j} +                                                                                                                                                                                                                                                                                                                                                                                                                                                                                                                                                                                                                                                                                                                                                                                                                                                                                                                                                                                                                                                                                                                   \sum_{i=2, j>i}^{q, q+1}(-1)^{j-i-1}T_{i,j}$$
and consider $\phi(A)$, where $\phi$ is the differential $\phi=\mathrm{Hom}(\theta_k(a,b),D_{a+k}\otimes\wedge^{b-k})$.

\begin{lm} We have $\phi(A)=(k+1)\Gamma_k$. Moreover, if $k+1=p^e$, $p$ prime, then $\pi(\Gamma_k)$ is a a generator of $\mathrm{Ext}^k(\Delta(\mathrm{h}), D_{a+k}\otimes\wedge^{b-k})$.

\end{lm}
\begin{proof} For the first statement, it suffices to show that
\begin{equation*}\phi_1(A)=\tbinom{a+k-1}{k}(a+k)\tr_{1} \;\; \hbox{and} \; \; \phi_t(A)=(k+1)\tr_{t}
	\end{equation*} 
 for $t=2,\dots,q$, since $\binom{a+k-1}{k}(a+k)=(k+1)\binom{a+k}{k+1}.$ Using Remark 2.2 an immediate calculation in each case shows the following.\\
 $ i=1 $: Then $\phi_1(T_{1,2})=(a+k)\tr_1$.\\
 $ i=2 $: Then $\phi_1(T_{2,j})=\binom{a+k-1}{k}\phi_1(T_{1,j})$ for all $j \ge 3.$\\
 $ i>2$: Then $\phi_1(T_{i,j})=0$ for all $j>i$.

Upon substituting, \begin{align*}\phi_1(A)=&\tbinom{a+k-1}{k}(a+k)\tr_1+\tbinom{a+k-1}{k}\sum_{j=3}^{q+1}(-1)^j\phi_1(T_{1,j}) +\sum_{j=3}^{q+1}(-1)^{j-3}\phi_1(T_{2,j})\\
=&\tbinom{a+k-1}{k}(a+k)\tr_1.	\end{align*}
Similarly for $t>1$, an immediate calculation in each case yields the following.\\$ i<t $: Then $\phi_t(T_{i,t})=\phi_t(T_{i,t+1})$ and $\phi_t(T_{i,j})=0$ if $j \neq t,t+1$.\\$ i=t $: Then $\phi_t(T_{t,t+1})=(k+1)\tr_t$ and $\phi_t(T_{t,j})=\phi_t(T_{t+1,j})$ if $j \ge t+2$.\\$ i\ge t+2 $: Then $\phi_t(T_{i,j})=0$ for all $j>i$.

Upon substituting, \begin{align*}\phi_t(A)=&\tbinom{a+k-1}{k}\left((-1)^t\phi_t(T_{i,t})+(-1)^{t+1}\phi_t(T_{i,t})\right)\\&+\sum_{i=2}^{t-1}\left((-1)^t\phi_t(T_{i,t})+(-1)^{t+1}\phi_t(T_{i,t})\right) \\&+(a+k)\tr_t+\sum_{j=t+2}^{q+1}\left((-1)^{j-t-1}\phi_t(T_{t,j})+(-1)^{j-t}\phi_t(T_{t,j})\right)\\
=& (a+k)\tr_t.\end{align*}

Let $k+1=p^e$, $p$ prime. By Lemma 4.3 and the first part of the present Lemma, $\pi(\Gamma_k$) is a nonzero torsion element of the abelian group $E^k(\Delta(\mathrm{h}), D_{a+k}\otimes\wedge^{b-k})$. Thus it is a nonzero element of $\mathrm{Ext}^{k}(\Delta(\mathrm{h}), D_{a+k}\otimes\wedge^{b-k})$ which according to Lemma 2.3 is $\mathbb{Z}_p$. Hence it is a generator of $\mathrm{Ext}^{k}(\Delta(\mathrm{h}), D_{a+k}\otimes\wedge^{b-k})$.\end{proof}
%%%%%%%%%%%%%%%%%%%%%%%%%%%%%%%%%%%%%%%%%%%%%%%%%%%%%%%%%%%%%

\subsection{Relations and a generator of $E^k(\Delta(\mathrm{h}),\Delta(\mathrm{h}(k)))$}
Let $q=b-k+1$ and define $\delta_1\in \mathrm{Hom}(D(a+k,1,\dots,1),\Delta(\mathrm{h}(k)))$ and $\delta_{i,j} \in \mathrm{Hom}(D(a+k-j,1,\dots,j+1,\dots ,1), \Delta(\mathrm{h}(k)))$, where $j+1$ is located at the i-th position, by
 \begin{align*}&\delta_1=1^{(a+k)}|2\dots q, \\ &\delta_{i,j} = 1^{(a+k-j)}i^{(j)}| 2  \dots q, \; i=2,\dots,q, \; j=0, \dots k,\end{align*}
 where it is understood that for $j=0$ we have $\delta_{i,0}=\delta_1.$
\begin{lm} In $E^k(\Delta(\mathrm{h}),\Delta(\mathrm{h}(k)))$ the following relations hold.
\[\pi(\delta_{i,j}) = \tbinom{a+k+i-2}{j} \pi(\delta_1), \; i=2,\dots q-1, j=0, \dots k. \]
\end{lm}
\begin{proof} Let $S_{i,j} \in \mathrm{Hom}(P_{k-1}(a,b),\Delta(\mathrm{h}(k)))$, $i=1,\dots,q$, $$S_{i,j}=1^{(a+k-j)}(i+1)^{(j)}|2\dots \widehat{(i+1)} \dots (q+1)$$ and consider the differential in degree $k$ \[\sum_{t \ge 1}(-1)^{t-1}\phi_t: \mathrm{Hom}(P_{k-1}(a,b),\Delta(\mathrm{h}(k)) \rightarrow \mathrm{Hom}(P_{k}(a,b),\Delta(\mathrm{h}(k)))\] of the complex $\mathrm{Hom}(P_*(a,b),\Delta(\mathrm{h}(k)))$. An immediate calculation in each case using Remark 2.2 (and the straightening law for the first equality in the case $i>1$) yields the following.
	\begin{alignat*}{2}i=1: \;&\phi_1(S_{1,j})=\tbinom{a+k}{j} \delta_1, \; \phi_2(S_{1,j})=\delta_{2,j}, \; \phi_j(S_{1,j})=0, \; j>2.\\
                       i>1: \;&\phi_1(S_i) ={(-1)}^{i-1} \delta_{i,j-1}, \; \phi_t(S_{i,j}) = 0, \; t\in \{2,\dots,i-1\},\\
                            &\phi_i(S_{i,j})=\delta_{i,j},\; \phi_{i+1}(S_{i,j})=\delta_{i+1,j},\; \phi_t(S_{i,j})=0, \;t \ge i+2.\end{alignat*}
		 
In $E^{k}(a,b,\Delta(\mathrm{h}(k)))$, we have the relations \[\sum_{t=1}^{q}(-1)^{t-1} \pi(\phi_t(S_{i,j}))=0, \; i=1,\dots q.\]
Substituting the above for $i=1$ yields  
\begin{equation}
\pi(\delta_{2,j}) = \tbinom{a+k}{j} \pi(\delta_1), j=0, \dots k\end{equation}
and substituting the above for $i \ge 2$  yields 
\begin{equation}
\pi(\delta_{i+1,j})=\pi(\delta_{i,j})+\pi(\delta_{i,j-1}), \; i=2,\dots q-1, j=1, \dots k. 
\end{equation}
The equation of the Lemma follows by induction on $i$ using (4.2) and (4.3).
\end{proof}

\begin{lm} $E^k(\Delta(\mathrm{h}),\Delta(\mathrm{h}(k)))$ is a cyclic group generated by $\pi(\delta_1)$.
\end{lm}
\begin{proof} Induction on $k$, the case $k=1$ owing to the first paragraph of the proof of Lemma 3.3.
	
	Let $T \in \mathrm{Hom}(P_{k}(a,b),\Delta(\mathrm{h}(k)))$ be a semi-standard tableau of weight $\lambda=(\lambda_1,...,\lambda_q)$, $q=b-k+1$. If $\lambda_1 >a $ then by induction and Lemma 4.2 (since $\mathrm{h}(k)=(a+1+(k-1), 1^{b-1-(k-1)}$), $\pi(T)$ is a multiple of $\pi(\delta_1)$. 
	
	Suppose $\lambda_1=a$ in which case $T = 1^{(a)}2^{(\lambda_2 -1)}\dots q^{(\lambda_{q} -1)} | 2 \dots q.$ (In this notation it is understood that if $\lambda_j=1$, then the term $j^{(\lambda_j-1)}$ is omitted.) We will show that there are semi-standard tableaux $T_i \in \mathrm{Hom}(P_{k}(a,b),\Delta(\mathrm{h}(k)))$ and $a_i \in \mathbb{Z}$ such that $\pi(T) =\sum_{j}a_i\pi(T_j)$ and $T_j<T$ for every $j$ in the lexicographic ordering of semi-standard tableaux of $\mathrm{Hom}(P_{k}(a,b),\Delta(\mathrm{h}(k)))$. Since this set is finite and $\delta_1$ is the least element, we obtain by induction on the ordering that $\pi(T)=a\pi(\delta_1)$, $a \in \mathbb{Z}$.
	
	There exists an $i \ge 2$ such that $\lambda_i \ge 2$ because $k>0.$ Let  $m$ be the largest such $i$ and let
	$$S=1^{(a)}2^{(\lambda_2 -1)}\dots m^{(\lambda_{m} -1)} | 2 \dots \widehat{m} \dots (q+1),$$
	which is a semi-standard tableau in $\mathrm{Hom}(P_{k-1}(a,b),\Delta(\mathrm{h}(k))).$ Then for $t\in \{1,\dots,q\}$, straightforward calculations yield the following, where we assume that $m \ge 3$.\begin{enumerate}
		\item $\phi_1(S)= (-1)^{m-2}\tbinom{a+\lambda_2-1}{\lambda_2-1}1^{(a+\lambda_2)}2^{(\lambda_3-1)}\dots (m-1)^{(\lambda_m-2)}|2\dots q.$
		\item $\phi_t(S)=0$ if $m \ge 4$, $t=2,...,m-2$.
		\item $\phi_{m-1}(S)= (-1)^m \tbinom{\lambda_{m-1}+\lambda_m-2}{\lambda_m-2} 1^{(a+\lambda_2)}2^{(\lambda_3-1)}\dots (m-1)^{(\lambda_{m-1}+\lambda_m-2)}|2\dots q.$
		\item
		$\phi_m(S)= {(-1)}^{m-1} T$.
		\item $\phi_t(S)=0$, if $j \ge m+1$.\end{enumerate}
	The tableaux in the right-hand sides of equations (1) and (3) are semi-standard and less than $T$ in our ordering since $\lambda_2>0$, and the coefficient of $T$ in the right hand side of (4) is $\pm1$. Hence the desired result for $m \ge 3$ follows from $\sum_{t=1}^q(-1)^{t-1}\pi\phi_t(S)=0$.
	
	Let $ m=2 $. Then similarly, $0=\sum_{t=1}^{q}(-1)^{t-1}\pi\phi_t(S)=\tbinom{a+\lambda_2-1}{\lambda_2-1}\pi(\delta_1)-\pi(T)$ and the result follows.\end{proof}

\subsection{Proof of Theorem 4.1}
We prove Theorem 4.1 by induction on $k$, the case $ k=1 $ owing to \cite{Ma}, Theorem 6. Applying $\mathrm{Hom}(\Delta(\mathrm{h}),-)$ to the short exact sequence 
$$0 \rightarrow \Delta(\mathrm{h}(k+1)) \xrightarrow {i_k} D_{a+k} \otimes \wedge^{b-k} \xrightarrow {\pi_k} \Delta(\mathrm{h}(k)) \rightarrow 0 $$ yields the exact sequence
\begin{align}\cdots &\rightarrow \mathrm{Ext}^k(\Delta(\mathrm{h}),D_{a+k} \otimes \wedge^{b-k}) \xrightarrow {\pi^*_k} \mathrm{Ext}^k(\Delta(\mathrm{h}),\Delta(\mathrm{h}(k))) \rightarrow \nonumber \\
&\rightarrow  \mathrm{Ext}^{k+1}(\Delta(\mathrm{h}),\Delta(\mathrm{h}(k+1))) \rightarrow 0
\end{align}
because from Lemma 2.3, $ \mathrm{Ext}^{k+1}(\Delta(\mathrm{h}),D_{a+k} \otimes \wedge^{b-k})= \mathrm{Ext}^{k+1}(\wedge^{k+1}, D_{k+1})=0$ as the length of the projective resolution $P_*(1,k)$ of $\wedge^{k+1}$ is less than  $k+1$. By Lemma 2.3 we have $\mathrm{Ext}^k(\Delta(\mathrm{h}),D_{a+k} \otimes \wedge^{b-k})=  \mathbb{Z}_{r_k}.$

If  $\mathrm{Ext}^k(\Delta(\mathrm{h}),\Delta(\mathrm{h}(k)))=0$, then by induction $d_k=1$. Since $d_{k+1}|d_k$, we have $ d_{k+1}=1.$ Moreover, $\mathrm{Ext}^{k+1}(\Delta(\mathrm{h}),\Delta(h(k+1)))=0$ by (4.4) and hence the result holds in this case.

We may assume that $\mathrm{Ext}^k(\Delta(\mathrm{h}),\Delta(\mathrm{h}(k)))\neq0$. Since  $E^k(\Delta(\mathrm{h}),\Delta(\mathrm{h}(k)))$ is a cyclic $\mathbb{Z}$-module according to Lemma 4.6, and its torsion subgroup is nonzero, we have $E^k(\Delta(\mathrm{h}),\Delta(\mathrm{h}(k)))$ $=\mathrm{Ext}^k(\Delta(\mathrm{h}),\Delta(\mathrm{h}(k)))$.\\
\textit{Case 1}: Let $k+1=p^e$, $p$ prime. By Lemma 4.4, $\pi(\Gamma_k)$ is a generator of $\mathrm{Ext}^k(\Delta(\mathrm{h}),D_{a+k} \otimes \wedge^{b-k})$. We compute its image under the map $\pi_k^*$ of (4.4). 
	With the notation established at the beginning of subsections 4.1 and 4.2 and using the straightening law and Lemma 4.5 we have 
	\begin{align*}
	\pi^*_k(\pi(\tr_i)) &= \pi(1^{(a-1)} i^{(k+1)}|1 \dots \widehat{i} \dots q)= {(-1)}^{i+1}\pi(\delta_{i,k})\\&={(-1)}^{i+1}\tbinom{a+k+i-2}{k}\pi(\delta_{1}).
	\end{align*}
	By substituting in (4.1) and using the binomial coefficient identity $$\tbinom{a+k}{k+1} + \sum_{i=2}^{q}\tbinom{a+k+i-2}{k} = \tbinom{a+b}{k+1},$$ we obtain $\pi^*_k (\Gamma_k) =\tbinom{a+b}{k+1} \pi(\delta_{1}).$ Since, by Lemma 4.6, $\pi(\delta_1)$ is a generator of  $\mathrm{Ext}^{k}(\Delta(\mathrm{h}),\Delta(\mathrm{h}(k)))$,  we obtain from (4.4) and the induction hypothesis that  $\mathrm{Ext}^{k+1}(\Delta(\mathrm{h}),\Delta(h(k+1)))=\mathbb{Z}_{d_{k+1}}$.\\
	\textit{Case 2}: Suppose $k+1$ is divisible by two distinct primes, whence according to Lemma 2.3, $\mathrm{Ext}^k(\Delta(\mathrm{h}),D_{a+k} \otimes \wedge^{b-k})=0$. From (4.4) it suffices to show that $ d_k=d_{k+1}.$ 

By Theorem 1 of \cite{JOS}, $d_k=\frac{a+b}{l_k}$, where $l_k=lcm(1^{\eta_1},2^{\eta_2}, \dots ,k^{\eta_k})$ and $\eta_i=1$ if $i | a+b$, and $\eta_i=0$ otherwise.
%$\eta_i=\begin{cases} 1, & \mbox{if}\; \mbox{$i | a+b$} \\ 0, &  %\mbox{otherwise}. \end{cases}$
If $k+1 \not| a+b$, then $\eta_{k+1} =0$ and hence $l_k=l_{k+1}$. If $k+1 | a+b$, then, since  $ k+1 $ is divisible by two distinct primes, every prime power factor of $k+1$ is less than $k+1$. Hence  $l_{k+1}=l_{k}$. \qed

Let $\mathbb{K}$ be an infinite field of characteristic $p>0$, $S_{\mathbb{K}}(n,r)$ the Schur algebra for $GL_n(\mathbb{K})$ and $\Delta_{\mathbb{K}}(\lambda)$ the Weyl module for $S_{\mathbb{K}}(n,r)$  corresponding to a partition $\lambda$ of $r$ with at most $ n $ parts. Then $S_{\mathbb{K}}(n,r)$ =$\mathbb{K}\otimes S(n,r)$ and $\Delta_{\mathbb{K}}(\lambda)$ =$\mathbb{K}\otimes \Delta(\lambda)$. From this and the universal coefficient theorem \cite{AB}, Theorem 5.3, our results yield the following.
\begin{cor}
Let $\mathbb{K}$ be an infinite field of characteristic $p>0$ and $n\ge b+1$.\begin{enumerate}
	\item Let $2 \le k \le b$. Then $\mathrm{Hom}_{S_{\mathbb{K}}(n,r)}(\Delta_{\mathbb{K}}(\mathrm{h}), \Delta_{\mathbb{K}}(\mathrm{h}(k)))=0$, unless $p=2$ and $ a+b+k $ is odd, in which case $\mathrm{Hom}_{S_{\mathbb{K}}(n,r)}(\Delta_{\mathbb{K}}(\mathrm{h})), \Delta_{\mathbb{K}}(\mathrm{h}(k)))=\mathbb{K}$.
	\item Let $1 \le k \le b$. Then $\mathrm{Ext}^{k}_{S_{\mathbb{K}}(n,r)}(\Delta_{\mathbb{K}}(\mathrm{h}), \Delta_{\mathbb{K}}(\mathrm{h}(k)))=0$, unless $p|\tbinom{a+b}{i}, i=1,...,k$, in which case $\mathrm{Ext}^{k}_{S_{\mathbb{K}}(n,r)}(\Delta_{\mathbb{K}}(\mathrm{h}), \Delta_{\mathbb{K}}(\mathrm{h}(k)))=\mathbb{K}$.
\end{enumerate}\end{cor}

\section{Acknowledgments} We thank H. H. Andersen for various helpful discussions and for pointing out an error in the Remark after Theorem 4.1 and suggesting a proof.  We thank the reviewer for
detailed constructive comments and suggestions that helped improve the presentation and clarity of the
paper and for pointing out an error in the Remark after Theorem 4.1 and suggesting a proof.

\end{document}